\documentclass[12pt]{article}
\usepackage{graphicx}
\usepackage{verbatim}
\usepackage{amsmath}
\usepackage{amssymb}  
\usepackage{url}  

\begin{document}

\font\tenrm=cmr10
\font\ninerm=cmr9
\font\eightrm=cmr8
\font\sevenrm=cmr7
\font\ss=cmss10 
\font\smallcaps=cmcsc10 
%
\overfullrule=0pt

%
%

\newcommand{\cldot}[1]{{{\rm CLDOT}({#1})}}
\newcommand{\dorb}{{\rm dorbit}}
\newcommand{\orb}{{\rm orbit}}
\renewcommand{\gets}{\leftarrow}
\newcommand{\cl}{{\rm cl}}
\newcommand{\BISEQ}{{\rm BISEQ}}
\newcommand{\RED}{{\rm RED}}
\newcommand{\BLUE}{{\rm BLUE}}
\newcommand{\GREEN}{{\rm GREEN}}
\newcommand{\Tinv}{T^-}
\newcommand{\T}[1]{T^{(#1)}}

\newcommand{\card}[1]{\#(#1)}
\newcommand{\st}{\mathrel{:}}
\newcommand{\fhat}{{\hat{f}}}
\newcommand{\Ahat}{{\hat{A}}}
\newcommand{\Bhat}{{\hat{B}}}
\newcommand{\Chat}{{\hat{C}}}
\newcommand{\goes}{\rightarrow}
\newcommand{\IO}{\exists^\infty}
\renewcommand{\AE}{\forall^\infty}

\newcommand{\Tbar}{\overline{T}}
\newcommand{\sm}{{\rm small}}
\newcommand{\bi}{{\rm big}}
\newcommand{\ts}{t^{\sm}}
\newcommand{\tb}{t^{\bi}}
\newcommand{\tsj}{t^{\sm,j}}
\newcommand{\tbj}{t^{\bi,j}}
\newcommand{\Bs}{B^{\sm}}
\newcommand{\Bb}{B^{\bi}}

\newcommand{\Abar}{\overline{A}}
\newcommand{\bit}{\{0,1\}}
\newcommand{\bits}[1]{\{0,1\}^{#1}}
\newcommand{\es}{\emptyset}
\newcommand{\CL}{\mathord{\mbox{\it CL}}}
\newcommand{\COL}{\chi}
\newcommand{\alphap}{{\alpha'}}
\newcommand{\lam}{\lambda}
\newcommand{\betap}{{\beta'}}
\newcommand{\aprime}{a'}

\newcommand{\linem}{line$^-$\ }
\newcommand{\linemns}{line$^-$}
\newcommand{\linems}{lines$^-$\ }
\newcommand{\linemsns}{lines$^-$}
\newcommand{\newa}{a}
\newcommand{\olda}{a'}
\newcommand{\olddd}{d'}
\newcommand{\oldd}[1]{d_{#1}'}
\newcommand{\newdd}{d_1}
\newcommand{\newd}[1]{d_{#1}}
\newcommand{\vdw}{{\rm VDW }}
\newcommand{\bvdw}{{\bf VDW }}
\newcommand{\bvdwns}{{\bf VDW}}
\newcommand{\vdwns}{{\rm VDW}}
\newcommand{\polyvdw}{{\rm POLYVDW }}
\newcommand{\bpolyvdw}{{\bf POLYVDW }}
\newcommand{\bpolyvdwns}{{\bf POLYVDW}}
\newcommand{\polyvdwns}{{\rm POLYVDW}}
\newcommand{\hj}{{\rm HJ }}
\newcommand{\bhj}{{\bf HJ }}
\newcommand{\bhjns}{{\bf HJ}}
\newcommand{\hjns}{{\rm HJ}}
\newcommand{\polyhj}{{\rm POLYHJ }}
\newcommand{\bpolyhj}{{\bf POLYHJ }}
\newcommand{\bpolyhjns}{{\bf POLYHJ}}
\newcommand{\polyhjns}{{\rm POLYHJ}}

%

\newcommand{\Z}{\mathbb{Z}}
\newcommand{\Zpd}{\mathbb{Z}_p^d}
\newcommand{\Zp}{\mathbb{Z}_p}
\newcommand{\Zn}{\mathbb{Z}_n}
\newcommand{\N}{\mathbb{N}}
\newcommand{\Q}{\mathbb{Q}}
\newcommand{\R}{\mathbb{R}}
\newcommand{\C}{\mathbb{C}}

\newcommand{\lf}{\left\lfloor}
\newcommand{\rf}{\right\rfloor}
\newcommand{\lc}{\left\lceil}
\newcommand{\rc}{\right\rceil}
\newcommand{\Ceil}[1]{\left\lceil {#1}\right\rceil}
\newcommand{\ceil}[1]{\left\lceil {#1}\right\rceil}
\newcommand{\floor}[1]{\left\lfloor{#1}\right\rfloor}

\newcommand{\nth}{n^{th}}
%
%
\newcommand{\into}{\rightarrow}
\newcommand{\inter}{\cap}
\newcommand{\union}{\cup}
\newcommand{\sig}[1]{\sigma_{#1} }
\newcommand{\s}[1]{\s_{#1}}
\newcommand{\LMA}{{\rm L}(M^A)}
\newcommand{\Ah}{{\hat A}}
\newcommand{\monus}{\;\raise.5ex\hbox{{${\buildrel
   \ldotp\over{\hbox to 6pt{\hrulefill}}}$}}\;}
\newcommand{\dash}{\hbox{-}}
\newcommand{\infinity}{\infty}
\newcommand{\ie}{\hbox{ i.e.  }}
\newcommand{\eg}{\hbox{ e.g.  }}
\newcommand{\divides}{|}
\newcommand{\notdivides}{\mathrel{\not |}}
\newcommand{\eps}{\varepsilon}
\newcommand{\vect}[1]{{\bf #1}}

%
%
%
%
%
%
\newcounter{savenumi}
\newenvironment{savenumerate}{\begin{enumerate}
\setcounter{enumi}{\value{savenumi}}}{\end{enumerate}
\setcounter{savenumi}{\value{enumi}}}
\newtheorem{theoremfoo}{Theorem}[section] 
\newenvironment{theorem}{\pagebreak[1]\begin{theoremfoo}}{\end{theoremfoo}}
\newenvironment{repeatedtheorem}[1]{\vskip 6pt
\noindent
{\bf Theorem #1}\ \em
}{}

\newtheorem{lemmafoo}[theoremfoo]{Lemma}
\newenvironment{lemma}{\pagebreak[1]\begin{lemmafoo}}{\end{lemmafoo}}
\newtheorem{conjecturefoo}[theoremfoo]{Conjecture}
\newtheorem{research}[theoremfoo]{Line of Research}
\newenvironment{conjecture}{\pagebreak[1]\begin{conjecturefoo}}{\end{conjecturefoo}}

\newtheorem{conventionfoo}[theoremfoo]{Convention}
\newenvironment{convention}{\pagebreak[1]\begin{conventionfoo}\rm}{\end{conventionfoo}}

\newtheorem{porismfoo}[theoremfoo]{Porism}
\newenvironment{porism}{\pagebreak[1]\begin{porismfoo}\rm}{\end{porismfoo}}

\newtheorem{gamefoo}[theoremfoo]{Game}
\newenvironment{game}{\pagebreak[1]\begin{gamefoo}\rm}{\end{gamefoo}}

\newtheorem{corollaryfoo}[theoremfoo]{Corollary}
\newenvironment{corollary}{\pagebreak[1]\begin{corollaryfoo}}{\end{corollaryfoo}}

\newtheorem{openfoo}[theoremfoo]{Open problem}
\newenvironment{open}{\pagebreak[1]\begin{openfoo}\rm}{\end{openfoo}}

\newtheorem{exercisefoo}{Exercise}
\newenvironment{exercise}{\pagebreak[1]\begin{exercisefoo}\rm}{\end{exercisefoo}}

\newcommand{\fig}[1] 
{
\begin{figure}
\begin{center}
\input{#1}
\end{center}
\end{figure}
}

\newtheorem{potanafoo}[theoremfoo]{Potential Analogue}
\newenvironment{potana}{\pagebreak[1]\begin{potanafoo}\rm}{\end{potanafoo}}

\newtheorem{notefoo}[theoremfoo]{Note}
\newenvironment{note}{\pagebreak[1]\begin{notefoo}\rm}{\end{notefoo}}

\newtheorem{notabenefoo}[theoremfoo]{Nota Bene}
\newenvironment{notabene}{\pagebreak[1]\begin{notabenefoo}\rm}{\end{notabenefoo}}

\newtheorem{nttn}[theoremfoo]{Notation}
\newenvironment{notation}{\pagebreak[1]\begin{nttn}\rm}{\end{nttn}}

\newtheorem{empttn}[theoremfoo]{Empirical Note}
\newenvironment{emp}{\pagebreak[1]\begin{empttn}\rm}{\end{empttn}}

\newtheorem{examfoo}[theoremfoo]{Example}
\newenvironment{example}{\pagebreak[1]\begin{examfoo}\rm}{\end{examfoo}}

\newtheorem{dfntn}[theoremfoo]{Def}
\newenvironment{definition}{\pagebreak[1]\begin{dfntn}\rm}{\end{dfntn}}

\newtheorem{propositionfoo}[theoremfoo]{Proposition}
\newenvironment{proposition}{\pagebreak[1]\begin{propositionfoo}}{\end{propositionfoo}}
\newenvironment{prop}{\pagebreak[1]\begin{propositionfoo}}{\end{propositionfoo}}

\newenvironment{proof}
   {\pagebreak[1]{\narrower\noindent {\bf Proof:\quad\nopagebreak}}}{\QED}
\newenvironment{proofof}[1]
   {\pagebreak[1]{\narrower\noindent {\bf Proof of #1:\quad\nopagebreak}}}{\QED}
\newenvironment{sketch}
   {\pagebreak[1]{\narrower\noindent {\bf Proof sketch:\quad\nopagebreak}}}{\QED}

\newcommand{\yyskip}{\penalty-50\vskip 5pt plus 3pt minus 2pt}
\newcommand{\blackslug}{\hbox{\hskip 1pt
       \vrule width 8pt height 8pt depth 1.5pt\hskip 1pt}}
\newcommand{\QED}{{\penalty10000\parindent 0pt\penalty10000
       \hskip 8 pt\nolinebreak$\square$\hfill\lower 8.5pt\null}
       \par\yyskip\pagebreak[1]}

\newcommand{\BBB}{{\penalty10000\parindent 0pt\penalty10000
       \hskip 8 pt\nolinebreak\hbox{\ }\hfill\lower 8.5pt\null}
       \par\yyskip\pagebreak[1]}
    
\newcommand{\PYI}{CCR-8958528}
\newtheorem{factfoo}[theoremfoo]{Fact}
\newenvironment{fact}{\pagebreak[1]\begin{factfoo}}{\end{factfoo}}
\newenvironment{acknowledgments}{\par\vskip 6pt\footnotesize Acknowledgments.}{\par}




\newenvironment{construction}{\bigbreak\begin{block}}{\end{block}
   \bigbreak}

\newenvironment{block}{\begin{list}{\hbox{}}{\leftmargin 1em
   \itemindent -1em \topsep 0pt \itemsep 0pt \partopsep 0pt}}{\end{list}}


\dimen15=0.75em
\dimen16=0.75em

\newcommand{\lblocklabel}[1]{\rlap{#1}\hss}

\newenvironment{lblock}{\begin{list}{}{\advance\dimen15 by \dimen16
   \leftmargin \dimen15
   \itemindent -1em
   \topsep 0pt
   \labelwidth 0pt
   \labelsep \leftmargin
   \itemsep 0pt
   \let\makelabel\lblocklabel
   \partopsep 0pt}}{\end{list}}


\newenvironment{lconstruction}[2]{\dimen15=#1 \dimen16=#2
 \bigbreak\begin{block}}{\end{block}\bigbreak}

\newcommand{\Comment}[1]{{\sl ($\ast$\  #1\ $\ast$)}}


\newcommand{\Erdos}{Erd\H{o}s}
\newcommand{\Szekely}{Sz\'ekely}
\newcommand{\Szemeredi}{Szemer{\'e}di}
\newcommand{\Holder}{\H\"older}

\title{Toward a graph version of Rado's theorem}
\author{Andy Parrish}
\maketitle

\begin{abstract}
\noindent
An equation is called {\it graph-regular} if it always has
monochromatic solutions under edge-colorings of $K_{\N}$.
We present two Rado-like conditions which are respectively
necessary and sufficient for an equation to be graph-regular.
\end{abstract}

\vspace{1pc} \noindent {\bf Keywords:} coloring; graph theory; Ramsey theory; Rado's theorem

\section{Introduction}

Ramsey's theorem \cite{GRS}\cite{ramsey} states that, given numbers $r$ and $k$,
there is an $R=R(r,k)$ so that any $r$-coloring of the edges of the
complete graph on $R$ vertices contains a monochromatic complete graph
on $k$ vertices.

Elsewhere in Ramsey theory, the related (but seemingly-unconnected)
result of Rado \cite{GKP}\cite{GRS}\cite{rado} characterizes the set
of linear equations $A \vect{x} = \vect{0}$ such that, whenever $\N$
is finitely colored, there is a solution whose entries fall into
the same color class. Such an equation is called {\it partition-regular}.

These two results were connected in a result by Deuber, Gunderson, Hindman,
and Strauss \cite{gunder1}, and followed up by Gunderson, Leader, Pr\"omel,
and R\"odl \cite{gunder2}\cite{gunder3}.  Their results describe the equations
$A \vect{x} = \vect{0}$ such that, for any $m$, a large two-colored graph must
either contain a complete blue subgraph whose vertices solve the equation,
or a red $K_m$ with no implied structure. If one can assure that there is
no red $K_m$, then the desired structured set falls out nicely.

The first example of an equation with an unconditional monochromatic solution
was given in \cite{parrish}. That paper introduced the notion of a
{\it graph-regular} equation.

\begin{definition}
We say $A \vect{x} = \vect{0}$ is graph-regular if there is a function
$N_A(r)$ so that, for all $r$, for all $N > N_A(r)$, every $r$-coloring of
the complete graph on $[N]$ has a solution $\vect{x} = (x(1), \ldots, x(k))$
so that (1) the edges $\{x(i), x(j)\}$ are all the same color, and
(2) the values $\{x(i)\}$ are distinct.
\end{definition}

We require a solution by distinct values due to degeneracy issues which do not
appear in the case of coloring points. Further, for non-triviality, we require
the equation to contain at least three variables.

In this paper, we give two extensions of Rado's ``columns condition'' to the
graph setting --- the {\it weak} and {\it strong} graph columns conditions.
We will show that the weak version is necessary for an equation to be
graph-regular, and the strong version is sufficient. We also show that
the notions of partition-regular and graph-regular equations do not have a
satisfying extension to hypergraph-regular equations.

\section{The Graph Columns Condition}\label{se:GCC}

\subsection{Definitions}

Rado's theorem \cite{GRS}\cite{rado} characterizes the regular equations
with use of the columns condition, which we state here.

\begin{definition}\label{de:CC}
A matrix $A$ with $n$ columns satisfies the columns condition if there is a
sequence of vectors $\vect{z}_1, \ldots, \vect{z}_T$ in the nullspace of $A$
and decreasing sequence of sets $R_1 \supseteq \ldots \supseteq R_T$
so that
\begin{enumerate}
\item If $i \in R_t$, then $z_s(i)=0$ for all $s \le t$.
\item If $i \notin R_t$, then there is an $s \le t$ with $z_s(i) = 1$
\item $R_T = \emptyset$.
\end{enumerate}
\end{definition}
Although it is not commonly stated in this language, the above is easily seen
to be equivalent to the usual definition.

The remarkable fact, proven for example in \cite{GRS}, is that the columns
condition determines whether an equation is regular.

\begin{lemma}\label{le:rado}
The equation $A \vect{x} = \vect{0}$ has a monochromatic solution under
any finite coloring of $\N$ whenever $A$ satisfies the columns condition.
If $A$ does not satisfy the columns condition then there is some
$p_0 = p_0(A)$ so that, for every prime $p > p_0$, a monochromatic solution
is avoided by the coloring $\psi_p$
(which will be introduced in Section~\ref{se:sumtozero}).
\end{lemma}

There are several ways to extend the columns condition to apply to
edge-colorings. We state two.

\begin{definition}\label{de:GCC}
We say a matrix $A$ with $n$ columns satisfies the
{\it weak graph columns condition} (WGCC) if there is a sequence of vectors
$\vect{z}_0, \ldots, \vect{z}_T$ in the nullspace of $A$,
and a decreasing sequence of graphs $R_0 \supseteq \ldots \supseteq R_T$ with
common vertex set $[n]$ so that
\begin{enumerate}
\item If $\{i, j\} \in R_t$, then $z_s(i)=z_s(j)$ for all $s < t$.
\item If $\{i, j\} \notin R_t$, then there is an $s \le t$ with
  $|z_s(j) - z_s(i)| = 1$.
\item $R_T$ is empty.
\item $\vect{z}_0 = \vect{1}$.
\end{enumerate}

Further, we say $A$ satisfies the {\it strong graph columns condition} (SGCC)
if we may replace (1) and (2) by ($1^*$) and ($2^*$):
\begin{enumerate}
\item[$1^*.$] If $\{i, j\} \in R_t$, $z_s(i)=z_s(j) \in \{0, 1\}$
  for all $s < t$
\item[$2^*.$] If $\{i, j\} \notin R_t$, then there is an $s \le t$ with
  $z_s(i) = 0$ and $z_s(j) = 1$ (or vice versa).
\end{enumerate}
\end{definition}

In words, the $\vect{z}_t$'s are restricted so that, for each $i, j$ pair,
as $t$ increases, the values $z_t(i), z_t(j)$ are initially equal, remain
equal until they differ by exactly 1, and are unrestricted after that.
An edge between $i$ and $j$ on graph $R_t$ means that pair remains restricted
through time $t$. If $\{i, j\} \in R_t$, then we say
the pair is {\it restricted} at time $t$, otherwise it is {\it unrestricted}.

The strong graph columns conditions requires conditions on the values
$z_t(i)$ in addition to the differences across edges.

\begin{example}\label{ex:6vars}
Let
\[
A = \left( \begin{array}{rrrrrr}
1 & -1 & 0 & -1 & 4 & -3\\
0 &  0 & 1 & -1 & 1 & -1\\
\end{array} \right).
\]
Here is one sequence of vectors showing $A$ satisfies
the graph columns condition (weak and strong):
\[
\vect{z}_0 = \left(
\begin{array}{c}
1 \\ 1 \\ 1 \\ 1 \\ 1 \\ 1 \\
\end{array}
\right) \qquad
\vect{z}_1 = \left(
\begin{array}{c}
1 \\ 1 \\ 0 \\ 0 \\ 0 \\ 0 \\
\end{array}
\right) \qquad
\vect{z}_2 = \left(
\begin{array}{c}
1 \\ 0 \\ 1 \\ 1 \\ 0 \\ 0 \\
\end{array}
\right) \qquad
\vect{z}_3 = \left(
\begin{array}{c}
3 \\ 0 \\ 1 \\ 0 \\ 0 \\ 1 \\
\end{array}
\right)
\]
The corresponding restriction graphs may be described simply:
\begin{itemize}
\item $R_1$: All edges among \{1, 2\} and \{3, 4, 5, 6\} remain restricted.
\item $R_2$: Edges \{3, 4\} and \{5, 6\} remain restricted.
\item $R_3$ is empty --- no edges are restricted.
\end{itemize}
Note that, at each step, $R_t$ is a union of
disjoint cliques. This always happens.
\end{example}

%
%

\begin{example}\label{ex:8vars}
Let
\[
A = \left( \begin{array}{rrrrrrrc}
1 & -1 & 0 &  0 & 0 &  0 & -1 &   1 \\
0 &  0 & 1 & -1 & 0 &  0 & -1 &   1 \\
0 &  0 & 0 &  0 & 1 & -1 & -1 &   1 \\
0 & -1 & 0 & -1 & 0 & -1 & -r & r+1 \\
\end{array} \right).
\]
Here is one sequence of vectors showing $A$ satisfies
the weak graph columns condition:
\[
\vect{z}_0 = \left(
\begin{array}{c}
1 \\ 1 \\ 1 \\ 1 \\ 1 \\ 1 \\ 1 \\ 1 \\
\end{array}
\right) \qquad
\vect{z}_1 = \left(
\begin{array}{c}
1 \\ 1 \\ 1 \\ 1 \\ 0 \\ 0 \\ 0 \\ 0 \\
\end{array}
\right) \qquad
\vect{z}_2 = \left(
\begin{array}{c}
1 \\ 1 \\ 0 \\ 0 \\ 1 \\ 1 \\ 0 \\ 0 \\
\end{array}
\right) \qquad
\vect{z}_3 = \left(
\begin{array}{c}
1 \\ 0 \\ 1 \\ 0 \\ 1 \\ 0 \\ r+1 \\ r \\
\end{array}
\right)
\]

Notice that $\vect{z}_3$ relaxes the restriction on the 7th and 8th columns
by using values $r$ and $r+1$, rather than 0 and 1 as required by the strong
graph columns condition.
\end{example}

We now state our main result:

\begin{theorem}\label{th:goal}
Fix a matrix $A$.
If $A \vect{x} = \vect{0}$ is graph-regular, 
then $A$ satisfies the weak graph columns condition.
If $A$ satisfies the strong graph columns condition,
then $A \vect{x} = \vect{0}$ is graph-regular.
\end{theorem}

We will prove WGCC is necessary in Section~\ref{se:necessary},
using Lemma~\ref{le:rado}.
In Section~\ref{se:sufficient}, we will show SGCC is sufficient, using
machinery from \cite{parrish}.

\section{Necessary conditions for graph-regularity}\label{se:necessary}

\subsection{Coefficients must add to 0}\label{se:sumtozero}
Rado's theorem tells us that, when coloring $\N$, an equation is regular
precisely if it has a monochromatic solution under a particular class of
colorings --- the so-called ``super mod $p$'' colorings \cite{GRS}.
The idea is to color each number by factoring out any power of $p$, and color
by what remains. We will use related colorings to avoid
monochromatic solutions to many equations in the graph setting.
The super mod $p$ coloring itself will be introduced as $\psi_p$ later on.

Consider $f_n$, a coloring of $\binom{\N}{2}$ given by
\[
f_n(an+x,bn+y) = \left\{
\begin{array}{ll}
\hbox{blue}  & \hbox{if } x = y     \\
\min \{x,y\} & \hbox{if } x \neq y, \\
\end{array}
\right.
\]
where $x,y \in \{0, 1, \ldots, n-1\}$.

Although strange, consider its purpose. We claim that all monochromatic
triangles under $f_n$ are blue. Consider a triangle $x = an+i,y=bn+j,z=cn+k$,
with $i,j,k \in \{0,1,\ldots,n-1\}$.

Suppose no edge among these points is blue. Then the numbers $i,j,k$
must be distinct. Reorder so that $i < j < k$.
Then we see that $f_n(x,y) = f_n(x,z) = i$, while $f_n(y,z)=j \neq i$.
Thus a triangle without a blue edge cannot be monochromatic. Turning this
around, any monochromatic triangle must have one, and hence all, blue edges.

Going back to the definition of $f_n$, this means that any monochromatic
triangle (and hence any non-trivial monochromatic clique) must
represent only one congruence class mod $n$.

This means any equation which has no solution within a congruence class
may be avoided by using $f_n$. The only problem now is that this doesn't
actually rule out any linear homogeneous equation --- they may all be
solved by multiples of $n$. By making a slight change, however, this
becomes very powerful.

Fix a prime $p$, and define the coloring $g_p$ of $\binom{\N}{2}$ by
\[
g_p(p^j a, p^k b) = \left\{
\begin{array}{ll}
\hbox{red}   & \hbox{if } j \neq k \\
f_{p-1}(a,b) & \hbox{if } j = k,   \\
\end{array}
\right.
\]
where $a$ and $b$ are not divisible by $p$. Note that this is a
$p$-coloring, since the numbers $a$ and $b$ are always in $[p-1]$,
so their minimum must be in $[p-2]$. Adding in red and blue, we reach
$p$ colors.

The same argument as above shows that any triangle which is monochromatic
under $g_p$ must be red or blue. This brings us to our first
general result.

\begin{lemma}\label{le:sumtozero}
If $A \vect{x} = \vect{0}$ is graph-regular, then $\sum \vect{a}_i = \vect{0}$.
\end{lemma}

\begin{proof}
Since each row of $A$ corresponds to a graph-regular equation,
it suffices to consider the case when $A$ has a single row.

Fix a prime $p$ and color $\binom{\N}{2}$ by $g_p$. Suppose
$x_1, \ldots, x_k$ are distinct values satisfying
$a_1 x_1 + \ldots + a_k x_k = 0$, with the edges among them
monochromatic. Write $x_i = p^{r_i} (b_i p + c_i)$.
As observed earlier, this clique must be either red or blue.

\smallskip
\noindent
{\bf Case 1}: The clique is red. This means each $r_i$ is distinct.
Let $r_j$ be the lowest exponent. We see that
\[
\begin{array}{rcl}
\displaystyle
0 & = & \displaystyle \sum a_i p^{r_i} (b_i p + c_i)\\
  & = & \displaystyle \sum a_i p^{r_i - r_j} (b_i p + c_i)\\
  & = & \displaystyle a_j (b_j p + c_j) + p
        \left( \sum_{i \neq j} p^{r_i - r_j - 1} (b_i p + c_i) \right)\\
  & \equiv & \displaystyle a_j c_j \mod{p}.\\
\end{array}
\]
Since $p$ does not divide $c_j$, it must divide $a_j$ instead.

\smallskip 
\noindent
{\bf Case 2}: The clique is blue. This means each $r_i = r$, and each $c_i = c$.
We see that
\[
\begin{array}{rcl}
0 &   =    & \displaystyle \sum a_i p^r (b_i p + c) \\
  &   =    & \displaystyle \sum a_i (b_i p + c)              \\
  & \equiv & \displaystyle c \left( \sum a_i \right) \mod{p}.\\
\end{array}
\]
Since $p$ does not divide $c$, it must divide $\sum a_i$ instead.

\bigskip
Since every $g_p$ must have a monchromatic solution, each $p$
gives a new divisor to one of $\{a_1, \ldots, a_k, \sum a_i\}$.
Therefore one of them must be divisible by infinitely many primes.
Since $a_1, \ldots, a_k$ are assumed to be non-zero, it must be that
$\sum a_i = 0$.
\end{proof}

Consider such an equation, $\vect{a}_1 x_1 + \ldots + \vect{a}_k x_k = 0$,
where the coefficients sum to 0. We may rewrite this as, for instance,
\[
\vect{a}_1 (x_1 - x_k) + \ldots + \vect{a}_{k-1} (x_{k-1} - x_k) = 0,
\]
now an equation relating differences.
This suggests that we should consider colorings based on these differences
--- colorings of the form $\chi(x < y) = f(y-x)$. We may now take guidance
from Rado's theorem to get a better handle on things.

For a prime $p$, and $x = p^r (bp + s)$, let $\psi_p(x) = s \in [p-1]$
be the ``super mod $p$'' coloring. Rado's theorem says that the regular
systems of linear equations are precisely those which have
monochromatic solutions under $\psi_p$ for all $p$.
This, he found, is equivalent to satisfying the ``columns condition,''
which we will discuss more later. For a single equation, it
simply requires that some (nonempty) subset of the coefficients add
to zero.

For us, Rado's theorem suggests that every equation which we can avoid
with this type of coloring (based only on differences between endpoints)
will also be avoided by some $\psi_p$. This is a powerful idea, but we
begin with a simple consequence.

\begin{theorem}\label{th:partition}
Let $\sum_{i=1}^k a_i x_i = 0$ be graph-regular.
Then there is a nonempty set $I \subsetneq [k]$ so that
\[
\sum_{i \in I} a_i = \sum_{j \notin I} a_j = 0.
\]
\end{theorem}

To prove this, we introduce an important coloring.
Define $\varphi_p : \binom{\N}{2} \rightarrow [p-1]$ by
\[
\varphi_p(x < y) = \psi_p(y - x).
\]

\begin{proof}
Fix a prime $p$ and color $\binom{\N}{2}$ by $\varphi_p$.
Suppose $x_1, \ldots, x_k$ are distinct values satisfying
$a_1 x_1 + \ldots + a_k x_k = 0$, with the edges among them all color $c$.
Let $x_j$ be the smallest of these values. As noted earlier, we see that
\[
\sum_{i \neq j} a_i (x_i - x_j) = 0.
\]
By choice of $x_j$, each of the terms $x_i - x_j$ is positive.
Thus we may write $x_i - x_j = p^{r_i} (b_i + c)$, since
$\varphi_p(x_j < x_i) = \psi_p(x_i-x_j) = c$. Let $r$ be the
smallest exponent among these $k-1$ terms, and let
$I = \{i \in [k]\setminus\{j\} \mid r_i = r\}$. Note that
$\emptyset \subsetneq I \subsetneq [k]$.
We see that
\[
\begin{array}{rcl}
0 & = & \displaystyle \sum_{i \neq j} a_i p^{r_i} (b_i p + c)\\
  & = & \displaystyle \sum_{i \neq j} a_i p^{r_i - r} (b_i p + c)\\
  & = & \displaystyle \sum_{i \in I} a_i (b_i p + c) +
        p \left( \sum_{i \notin I \cup \{j\}}
                 a_i p^{r_i - r - 1} (b_i p + c) \right)\\
  & \equiv & \displaystyle c \left( \sum_{i \in I} a_i \right) \mod{p}\\
\end{array}
\]
Since $c$ is in $[p-1]$, we see that $p$ divides $\sum_{i \in I} a_i$.
If we take $p > \sum_{i=1}^k |a_i|$, then the only way
this can happen is if $\sum_{i \in I} a_i = 0$. Since we already know that
$\sum_{i=1}^k a_i = 0$, we learn that $\sum_{i \notin I} a_i = 0$ as well.
\end{proof}

We may now prove what we already knew:

\begin{corollary}
No nondegenerate homogeneous linear equation of three variables
is graph-regular.
\end{corollary}

\begin{proof}
Let $k = 3$, and let $I \subsetneq \{1,2,3\}$ be nonempty. Then either
$I$ or its complement has a single element. The corresponding coefficient
must be 0, meaning the equation depends on at most two variables.
\end{proof}

\begin{corollary}
Up to rescaling, the only graph-regular homogeneous linear equation
of four variables is $w-x+y-z=0$.
\end{corollary}

That this equation is regular was shown in \cite{parrish}. We show that
it stands alone.

\begin{proof}
Let $aw+bx+cy+dz=0$ be graph-regular, with $a,b,c,d \in \Z_{\neq 0}$.
By Theorem~\ref{th:partition}, we know that two complementary subsets
of the coefficients must add to zero. Up to permutation, this leaves
us with $aw-ax+cy-cz=0$, or rather $a(w-x)=c(z-y)$. We may assume both
$a$ and $c$ are positive by switching $w$ and $x$, or $y$ and $z$.
We claim $a = c$.

Suppose not; without loss of generality, $a < c$. By cancelling
common divisors, we may assume $a$ and $c$ are relatively prime. Hence
$c = p_1^{r_1} \cdots p_k^{r_k}$, where none of these primes divides $a$.
Let $n = r_1 + 1$.
Define an $n$-coloring $\chi:\binom{\N}{2} \rightarrow \Z/n\Z$ by
\[
\chi(x,y) \equiv k \pmod{n} \quad \hbox{ if } (x-y)
  \hbox{ is divisible by } p_1^k, \hbox{ but not by } p_1^{k+1}.
\]
Now suppose $a(w-x) = c(z-y)$, with $w,x,y,z$ distinct.
Let $\chi(w,x) = k$. This means that $a(w-x)$ represents a power of
$p_1^k$ on the left hand side. Dividing by $c$, we learn that
$(z-y)$ represents a power of $p_1^{k-r_1}$,
so $\chi(y,z) \equiv k-r_1 \equiv k+1 \not\equiv k \pmod{n}$.
Thus $\chi(w,x) \neq \chi(y,z)$, so the edges among $\{w,x,y,z\}$
are not monochromatic.
\end{proof}

\subsection{Considering $\varphi_p$}

\begin{lemma}\label{le:phip}
Let $A$ be a matrix whose columns sum to $\vect{0}$.
If the equation $A \vect{x} = \vect{0}$ has a monochromatic solution under
the edge-coloring $\varphi_p$ for every prime $p$, then $A$ satisfies the
weak graph columns condition.
\end{lemma}

\begin{proof}
Let $A \vect{x} = \vect{0}$ have a monochromatic coloring under $\varphi_p$ for
every prime $p$. Denote the columns of $A$ by $\{ \vect{a}_i \}_{i=1}^n$.

From $A$, we will make a larger matrix $C$ with columns indexed by
$\binom{[n]}{2} = \{ (i,j) \mid 1 \le i < j \le n \}$.
The columns of $C$ come from gluing the columns of $A$ (or the zero vector)
to new vectors which bind relationships between the columns of $A$.

\[
\vect{c}_{1j} = \left( \begin{array}{c}
\vect{a}_j \\
\hbox{---} \\
\vect{b}_{1j}
\end{array} \right),
\hbox{ and }
\vect{c}_{ij} = \left( \begin{array}{c}
\vect{0} \\
\hbox{---} \\
\vect{b}_{ij}
\end{array} \right)
\hbox{ if } i > 1,
\]
where
\[
b_{ij}(k,\ell) = \left\{ \begin{array}{ll}
 1 & \hbox{if } (k,\ell) = (1,j) \\
-1 & \hbox{if } (k,\ell) = (1,i) \hbox{ or } (i,j)\\
 0 & \hbox{otherwise.}
\end{array} \right.
\]

Note that the matrix $C$ does not explicitly contain the column
$\vect{a}_1$. However, since $\sum \vect{a}_i = \vect{0}$, that information
is not lost.

Suppose $C \vect{y} = \vect{0}$, with $y(1,j) = x(j) - x(1)$. 
The vectors $\{ b_{ij} \}$ are designed so that $y(i,j) = x(j) - x(i)$.

Turned around, when $A \vect{x} = \vect{0}$, and $\vect{y}$ is defined above,
we get $C \vect{y} = \vect{0}$. Likewise, if $C \vect{y} = \vect{0}$ then,
for any value $x(1)$, the values $x(i)$ are uniquely defined from
$\vect{y}$, and they satisfy $A \vect{x} = \vect{0}$.

We would like to say that, when $A \vect{x} = \vect{0}$ is a monochromatic
solution under the edge-coloring $\varphi_p$, the corresponding solution to
$C \vect{y} = \vect{0}$ is monochromatic under the vertex-coloring $\psi_p$.
However, this is not quite true.
The definition says $\varphi_p(x,y) = \psi_p(y-x)$ only when $x < y$.
For a monochromatic solution under $\psi_p$, we would need
$x(1) < x(2) < \ldots < x(n)$.
Instead, for each permutation $\sigma \in S_n$, we must define the matrix
$C(\sigma)$ which will ``work'' when
$x(\sigma(1)) < x(\sigma(2)) < \ldots < x(\sigma(n))$.
We omit the definition of $C(\sigma)$, but it is essentially the same as $C$,
defined in such a way that $y(i,j)$ is always a positive number
when $\vect{x}$ is ordered by $\sigma$.

If $\vect{x}$ is a solution to $A \vect{x} = \vect{0}$, with
$x(\sigma(1)) < x(\sigma(2)) < \ldots x(\sigma(n))$,
then there is a corresponding solution to $C(\sigma) \vect{y} = \vect{0}$
by positive numbers, where $y(i,j) = x(\sigma(j)) - x(\sigma(i))$.
When $\vect{x}$ is monochromatic under $\varphi_p$,
$\vect{y}$ is monochromatic under $\psi_p$.

Claim: some $C(\sigma)$ satisfies the columns condition.

Proof: If not, then Lemma~\ref{le:rado} says each $\sigma$, gives a value
$p_0(C(\sigma))$ so that, for $p > p_0(C(\sigma))$ prime, $C(\sigma)$ has no
monochromatic solutions under $\psi_p$.
Let $p_0 = \max_{\sigma \in S_n} \{ p_0(C(\sigma)) \}$. Take a prime $p > p_0$.
Since $A \vect{x} = \vect{0}$ has a monochromatic solution under $\varphi_p$,
there must be some $\sigma \in S_n$ so that
$C(\sigma)$ has a monochromatic solution under $\psi_p$.
As $p > p_0(C(\sigma))$, this is a contradiction.

\bigskip \noindent
Fix $\sigma$ so that $C(\sigma)$ satisfies the columns condition.
This means there are vectors
$\vect{z}_1, \ldots, \vect{z}_T$ in the nullspace of $C(\sigma)$,
and decreasing graphs $R_1 \supseteq \ldots \supseteq R_T = \emptyset$
satisfying:
\begin{enumerate}
\item If $i \in R_t$, then $z_s(i)=0$ for all $s < t$.
\item If $i \notin R_t$, then there is an $s \le t$ with $z_s(i) = 1$.
\item $R_T = \emptyset$.
\end{enumerate}
For simplicity, we reorder the columns of $A$ so that $\sigma$ is the
identity, and $C(\sigma)$ is the matrix $C$ described originally.

This means there are vectors $\vect{w}_1, \ldots, \vect{w}_T$
indexed by $\binom{[n]}{2}$,
and sets $R_1 \supseteq \ldots \supseteq R_T = \emptyset$ with vertex set
$\binom{[n]}{2}$ satisfying conditions (1)-(3) of Definition~\ref{de:CC}.

Define a sequence of vectors $\vect{z}_1, \ldots, \vect{z}_T$ on $[n]$
by $z_t(1) = 0$, and $z_t(i) = w_t(1,i)$ for $i > 1$.
Additionally define $\vect{z}_0 = \vect{1}$ and $R_0 = [n]$.
We just need $\{ \vect{z}_t \}, \{ R_t \}$ to satisfy
requirements (1)-(4) of the graph columns condition.

It will be helpful to know that, for $k < \ell$,
\begin{equation}\label{eq:difference}
z_t(\ell) - z_t(k) = w_t(k,\ell)
\end{equation}
To see this, consider the $\{k,\ell\}$ row of the vectors $\vect{b}_{ij}$ within $C$.
Since $C \vect{w} = \vect{0}$, inspecting this row tells us that
\[
w_t(1,\ell) - w_t(1,k) = w_t(k,\ell).
\]
By definition of $\vect{z}_t$, we see that
\[
z_t(\ell) - z_t(k) = w_t(k,\ell)
\]
as desired.

Using this equation, properties (1)-(3) are immediate. Property (4) comes from
the assumption that the columns of $A$ sum to $\vect{0}$.
\end{proof}

\begin{corollary}
If a matrix $A$ is graph-regular, then it satisfies the weak graph columns
condition.
\end{corollary}

\begin{proof}
From Lemma~\ref{le:sumtozero}, we know that the columns of $A$ sum to
$\vect{0}$. Since $A$ is graph-regular, it must have a monochromatic solution
under $\varphi_p$ for every prime $p$. By Lemma~\ref{le:phip},
$A$ satisfies the weak graph columns condition.
\end{proof}

We end this section by considering the sufficiency of the WGCC.

\begin{corollary}
If a matrix $A$ satisfies the weak graph columns condition but is not
graph-regular, then the offending coloring is {\it not} of the form
$\chi(x < y) = f(y-x)$.
\end{corollary}

To see this in action, consider the coloring $\varphi_p$ and the matrix
$A$ from Example~\ref{ex:8vars}. Suppose that $1 < r < p^k-1$.
Consider the vector $\vect{x}$ in the nullspace of $A$ given by
\[
\vect{x} = (p^{k+2} + 1) \vect{z}_1 + (p^{k+1} + p) \vect{z}_2 + p^2 \vect{z}_3
\]
where $\{ vect{z}_i \}$ come from the analysis of this example earlier.
It is easy to check that $x(1) > x(2) > \ldots > x(8)$, and that $\varphi_p$
colors all edges by ``1''.
Indeed, Lemma~\ref{le:phip} actually shows that all colorings based on the
difference of the endpoints will yield a monochromatic solution.
Therefore, if the equation $A \vect{x} = \vect{0}$ is not graph-regular,
it must be from some other type of coloring.

\section{Sufficient conditions for graph-regularity}\label{se:sufficient}

We know that all graph-regular equations satisfy the weak graph columns
condition.
We now prove the strong graph columns condition. In order to do this,
we first define a large, hierarchical parametrized grid.

\begin{definition}
Fix $\vect{x}, \vect{y}, \vect{b}, \vect{d} \in \N^n$. We say the grid of
depth $n$ with parameters $\vect{x}, \vect{y}, \vect{b}, \vect{d}$
--- abbreviated $Grid_n(\vect{x}, \vect{y}, \vect{b}, \vect{d})$ ---
is the collection of points $A = \bigcup_{k=1}^n A_k$ where $A_k$ is the
set of points of the form
\[
\left( \sum_{\ell=1}^{k-1} h(\ell)d(\ell) + x(k)d(k) + i d(k+1),
       \sum_{\ell=1}^{k-1} h(\ell)d(\ell) + y(k)d(k) + j d(k+1) \right)
\]
such that $h(\ell) \in \{x(\ell), y(\ell)\}$ and $i,j \in [-c(k), c(k)]$
for some $c(k) \ge b(k)$. As it is not earlier defined, we use $d(n+1)=0$.

For such a grid to behave nicely, we always require the parameters to
satisfy:
\begin{enumerate}
\item $d(k) \divides d(k+1)$,
\item $x(k)b(k), y(k)b(k) \le c(k-1)$.
\end{enumerate}
\end{definition}

As a technical note, the use of $c(k) > b(k)$ comes from requirement (2).

We say that a $Grid_n(\vect{x}, \vect{y}, \vect{b}, \vect{d})$
is ``proper'' if (1) all $x$-coordinates of each of its points is less than
all $y$-coordinates, and (2) each coordinate has a unique
representation of the form
\[
\sum_{\ell=1}^{k} h(\ell)d(\ell) + i d(k+1)
\]
over all values of $k \in [n]$, $h(\ell) \in \{x(\ell), y(\ell)\}$,
and $i, j \in [-c(k), c(k)]$. We may thus unambiguously say that a point
``resides at the $k^{th}$ level'' of the grid if its coordinates agree on
$h(\ell)$ for all $\ell < k$, but disagree on $h(k)$.

For convenience of notation, we will treat a proper $Grid_n$ as a graph,
since it uniquely stores pairs $\{x, y\}$.

The proof of Theorem 3.1 in \cite{parrish} may be easily modified to give
the following lemma.

\begin{lemma}\label{le:monogrid}
Fix $\vect{b} \in \N^n$. There is a number $Q = Q(r, \vect{b})$ so that
every $r$-coloring of $[Q]^2$ admits vectors
$\vect{x}, \vect{y}, \vect{d} \in \N^n$ such that
$Grid_n(\vect{x}, \vect{y}, \vect{b}, \vect{d})$ is proper and monochromatic.

Moreover, the dependencies are such that the value of $b(k)$ may be a function
of upper bounds for $x(k+1), y(k+1), c(k+1),$ and $\frac{d(k+1)}{d(k)}$.
\end{lemma}

This lemma tells that we can always find a ``large'' monochromatic
$Grid_n$.

\subsection{A $Grid_n$ is enough}

Since we know every finite-coloring of $[Q] \times [Q]$ contains a
large monochromatic $Grid_n$ (for $Q$ sufficiently large),
we only need to show the following.

\begin{lemma}\label{le:solningrid}
Let $A$ satisfy the strong graph columns condition. Then there is some $n$,
$\vect{b}$ so that the following holds. For every every proper
$G = Grid_n(\vect{x}, \vect{y}, \vect{b}, \vect{d})$, there is a solution to
$A \vect{w} = \vect{0}$ so that, for all $i,j$, the edge
$\{w(i), w(j)\}$ is in $G$.

In particular, if $A$ satisfies the strong graph columns condition
in $T$ steps, then we may take $n = T$.
\end{lemma}

\begin{proof}
Let $A$ satisfy the columns condition, by vectors
$\vect{z}_0=\vect{1}, \vect{z}_1, \ldots, \vect{z}_T$ and graphs
$R_0 \supseteq \ldots \supseteq R_T = \emptyset$.

Fix $\vect{x}, \vect{y}, \vect{d} \in \N^T$. Define a sequence of vectors by
\[
\vect{v}_t = x(t)d(t) \vect{z}_0 + (y(t) - x(t))d(t) \vect{z}_t
           = x(t)d(t) \vect{1}   + (y(t) - x(t))d(t) \vect{z}_t,
\]
each in the nullspace of $A$.

Define $\vect{w} = \sum_{t=1}^{T} \vect{v}_t$. As a sum of
vectors in the nullspace of $A$, we have $A \vect{w} = \vect{0}$.
We claim that this is the desired solution.
It remains to show show that (1) every edge $\{w(i), w(j)\}$ is in $G$, and
(2) the values $w(i)$ are distinct.

Fix two indices $i$ and $j$. By the strong columns condition, we know that
the values $z_t(i)$ and $z_t(j)$ are initially equal --- with common value
0 or 1 --- when $\{i, j\} \in R_t$. There is a first time $t^*$ such that
(without loss of generality) $z_{t^*}(i)=0$ and $z_{t^*}(j)=1$.

Moving to $\vect{w}$, we see that
$z_t(i)=0$ contributes $x(t)$ to $w(i)$, and 
$z_t(i)=1$ contributes $y(t)$ to $w(i)$.
Since $v_t(i)$ and $v_t(j)$ agree for $t < t^*$, we may call the common
contribution at the $t$ step $h(t) = x(t)$ or $y(t)$. At time $t^*$,
$v(i) = x(t^*)$ while $v(j) = y(t^*)$. This suggests the edge
$\{ w(i), w(j) \}$ should reside at the $(t^*)^{th}$ level of $G$.
For $t > t^*$, we have
\[
\renewcommand{\arraystretch}{1.5}
\begin{array}{rcll}
|v_t(i)| &  =  & | x(t)d(t) + (y(t) - x(t))d(t)z_t(i) | \\
         &  =  & | x(t) + (y(t) - x(t))z_t(i) | d(t) \\
         & \le & | y(t) + y(t) z_t(i) | d(t)
                   & \hbox{since } 0 \le x(t) \le y(t) \\
         & \le & \left( \|z_t\|_{\infty} + 1 \right) y(t) d(t) \\
\end{array}
\renewcommand{\arraystretch}{1}
\]
Thus we see the total contribution to $w(i)$ from the tail of the sum comes to
\[
\renewcommand{\arraystretch}{1.5}
\begin{array}{rcll}
\displaystyle \left| \sum_{s > t^*} v_t(i)\right |
    & \le & \displaystyle \sum_{s > t^*} \left( \|z_s\|_{\infty} + 1 \right) y(s) d(s) \\
    & \le & \left( \|z_{t^*+1}\|_{\infty} + 1 \right) y(t^*+1) d(t^*+1) +
            \ldots + \\
    &     & + \left( \|z_{T-1}\|_{\infty} + 1 \right) y(T-1) d(T-1) + 
            \left( \|z_{T}\|_{\infty} + 1 \right) y(T) d(T)\\
\end{array}
\renewcommand{\arraystretch}{1}
\]
We now begin to see an appropriate choice of $\vect{b}$. Set
\[
b(T) \ge \|z_{T}\|_{\infty} + 1.
\]
The last term of our bound above becomes
$b(T) y(T) d(T)$, which by assumption is less than $c(T-1) d(T)$.

Next take
\[
b(T-1) \ge \left( \|z_{T-1}\|_{\infty} + 1 \right) y(T-1) +
           c(T-1) \left( \frac{d(T)}{d(T-1)} \right).
\]

Now the last two terms of the sum are bounded by $c(T-2) d(T-1)$.

We may continue this process so that, for $t > t^*$, we have
\[
b(t) \ge \left( \|z_{t}\|_{\infty} + 1 \right) y(t) +
         c(t) \left( \frac{d(t+1)}{d(t)} \right).
\]

Working backwards to step $t^* + 1$, we get
\[
\displaystyle \left| \sum_{t > t^*} v_t(i)\right | \le c(t^*) d(t^* + 1).
\]

We have written
\[
\renewcommand{\arraystretch}{1.5}
\begin{array}{rcll}
w(i) & = & h(1)d(1) + \ldots + h(t^* - 1) d(t^* - 1) + x(t^*) d(t^*) +
           p d(t^* + 1) \\
w(j) & = & h(1)d(1) + \ldots + h(t^* - 1) d(t^* - 1) + y(t^*) d(t^*) +
           q d(t^* + 1) \\
\end{array}
\renewcommand{\arraystretch}{1}
\]
with $p, q \in [-c(t^*), c(t^*)]$. Thus we see $\{w(i), w(j)\}$ is indeed
a point in $G$, in the $(t^*)^{th}$ level. 

To wrap up this argument, we need only take $b(t)$ to be the maximum of the
lower bounds seen, over all choices of $i$ and $j$. 

Finally, showing the points are distinct is simple. Since
$z_{t^*}(i) = 0$ and $z_{t^*}(j) = 1$, we see that $w(i)$ involves $x(t^*)$,
while $w(j)$ involves $y(t^*)$. Since the grid is proper, the values $w(i)$
and $w(j)$ must be distinct.
\end{proof}

\begin{corollary}
Let $A$ satisfy the strong graph columns condition.
Then $A \vect{x} = \vect{0}$ is graph-regular.
\end{corollary}

\begin{proof}
Let $\vect{c}$ be the vector given in Lemma~\ref{le:solningrid},
and let $r \in \N$ be a number of colors.
We claim that, if $Q \ge Q(r, \vect{c})$ from Lemma~\ref{le:monogrid},
then any $r$-coloring of $\binom{[Q]}{2}$ will contain a solution to
$A \vect{x} = \vect{0}$ so that the values $\{x(i)\}$ are distinct,
and the edges $\{ x(i), x(j) \}$ are monochromatic.

Indeed, by Lemma~\ref{le:monogrid}, viewing $\chi$ as an $r$-coloring of
$[Q] \times [Q]$ (minus the diagonal), we find a monochromatic $Grid_n$
of distinct points, with size at least $\vect{c}$.
By Lemma~\ref{le:solningrid}, this $Grid_n$ contains a solution to
$A \vect{x} = \vect{0}$ as desired.
\end{proof}

\section{Hypergraph-regular equations}

There is a natural extension of graph-regularity to the hypergraph Ramsey
theorem.

Unfortunately, this extension is not fruitful.
Say a homogeneous linear equation is ``$r$-graph-regular'' if, for every
coloring of the $r$-sets of $\N$, it has a monochromatic solution by
distinct numbers. As with graphs, when considering an $r$-uniform hypergraph,
we require the equations to have at least $r+1$ variables, or else every
solution will be trivially monochromatic.

\begin{theorem}
For $r \ge 3$, no homogeneous linear equation of at least $r+1$ variables
is $r$-graph-regular for $r$-uniform hypergraphs.
\end{theorem}

\begin{proof}
We show the result for $r = 3$, and suggest the appropriate modifications
for higher $r$.

Assume each $a_i$ is nonzero, since discarding trivial variables only makes
it easier to be graph-regular.

For any $n$, define an $(n+1)$-coloring $f_n^{(3)}$ of $\binom{\N}{r}$ by
\[
f_n^{(3)}(an+x,bn+y,cn+z) = \left\{
\begin{array}{ll}
\hbox{blue}  & \hbox{if } x=y=z \\
\min \{x,y,z\} & \hbox{if one of } {x,y,z} \hbox{ is smallest} \\
\max \{x,y,z\} & \hbox{otherwise,} \\
\end{array}
\right.
\]
where $x,y,z \in \{0, 1, \ldots, n-1\}$. Similar to before, any set of four
elements which is monochromatic under this coloring must be blue.

Now, for a prime $p$, define $g_p^{(3)}$ on $\binom{\N}{r}$ by
\[
g_p^{(3)}(p^i a, p^j b, p^k c) = \left\{
\begin{array}{ll}
f_{p-1}^{(3)}(a,b,c) & \hbox{if } i=j=k \\
\hbox{red}     & \hbox{if one of } {i,j,k} \hbox{ is smallest} \\
\hbox{green}   & \hbox{otherwise,} \\
\end{array}
\right.
\]
where $a,b,c$ are not divisible by $p$. Again, similar to before, any
monochromatic clique under this coloring on at least four points must
be red or blue. The proof of Lemma~\ref{le:sumtozero} now applies unchanged
to show that the coefficients of a hypergraph-regular equation must sum
to zero.

Therefore, we only consider $\sum_{i=1}^k a_i x_i = 0$ where
$\sum a_i = 0$.

Define a new coloring, $h_p(x,y,z) = g_p(y-x,z-x)$, where
$x < y < z$, and $g_p$ is the graph-coloring used in
Section~\ref{se:sumtozero}.

Suppose $x_1, \ldots, x_k$ are distinct values satisfying
$\sum a_i x_i = 0$, with the hyperedges among them monochromatic ---
either red or blue.
Let $x_j$ be the smallest of these values. Since
$a_j = -\sum_{i \neq j} a_i$, we see that
\[
\sum_{i \neq j} a_i (x_i - x_j) = 0.
\]
By choice of $x_j$, we see that $\{x_i - x_j\}_{i \neq j}$ is monochromatic
under $g_p$. As before, a red clique means some $a_i$ is 0.
If the clique is blue, then $\sum_{i \neq j} a_i = 0$, meaning $a_j = 0$.
Since none of the coefficients are 0, we have reached a contradiction.
Thus no homogeneous linear equation in at least 4 variables is
hypergraph-regular under colorings of 3-sets.

\bigskip
\noindent
For a general $r$-uniform hypergraph with $r > 3$, one can easily modify
the definition of $g_p^{(3)}$ to find a suitable $g_p^{(r)}$,
which will force coefficients to add to zero.
Likewise, one may define a coloring similar to $h_p$ which is built
upon $g_p^{(r-1)}$, which will force one of the coefficients to
be zero. These two colorings together will avoid solutions to
any equation in at least $r+1$ variables.
\end{proof}

Evidently, the ability to color 3-sets (or higher) of integers is too strong
to permit monochromatic solutions to homogeneous linear equations.
Is there a better definition for an equation to be $r$-graph-regular which
allows some equations to meet it, or is this the end of the story?


\end{document}